\documentclass[a4paper, DIV=12, 11pt]{amsart}
\usepackage[latin1]{inputenc}
\usepackage{amssymb, amsthm, amsmath, thmtools, mathtools}
\usepackage{amsfonts}
\usepackage[abbrev]{amsrefs} %modifies citations in theorems

\usepackage{graphicx}
\usepackage{thmtools}
\usepackage{hyperref}
\usepackage[bottom, marginal]{footmisc}
\usepackage{todonotes}

\declaretheoremstyle[headfont=\normalfont]{normalhead}
\newtheorem{lemma}{Lemma}[section]
\newtheorem{theorem}[lemma]{Theorem}
\newtheorem{proposition}[lemma]{Proposition}
\newtheorem{corollary}[lemma]{Corollary}
\newtheorem{definition}[lemma]{Definition}
\newtheorem{remark}[lemma]{Remark}

\newtheorem*{acknowledgement}{Acknowledgement}

\newcounter{mt}

\newtheorem{maintheorem}[mt]{Theorem}

\newcommand{\R}{\mathbb{R}}

\DeclareMathOperator{\vol}{vol}

\DeclareMathOperator{\diam}{diam}

\DeclareMathOperator{\GL}{GL}

\DeclareMathOperator{\Aff}{\mathrm{Aff}}

\DeclareMathOperator{\SO}{\mathrm{SO}}

\DeclareMathOperator{\SL}{\mathrm{SL}}

\newcommand{\calP}{\mathcal{P}}

\newcommand{\K}{\mathcal{K}}

\renewcommand{\S}{\mathbb{S}}
\DeclareMathOperator{\CV}{CV}

\numberwithin{equation}{section}

\author{Jonas Knoerr}
\title{Isometry invariant valuations on spherical polytopes}
\date{}

\newcommand{\Addresses}{{% additional braces for segregating \footnotesize
		\bigskip
		\footnotesize
		
		Jonas Knoerr, \textsc{Institute of Discrete Mathematics and Geometry, TU Wien, Wiedner Hauptstrasse 8-10, 1040 Wien, Austria}\par\nopagebreak
		\textit{E-mail address}: \texttt{jonas.knoerr@tuwien.ac.at}
		
		\medskip
	}}
	
\makeatletter
\def\blfootnote{\xdef\@thefnmark{}\@footnotetext}
\makeatother

\makeindex
\begin{document}
\maketitle
\begin{abstract}
	We show that every continuous and isometry invariant valuation on spherical polytopes is a linear combination of the spherical intrinsic volumes. The proof relies on a weak differentiability property satisfied by valuations on polytopes in $\mathbb{R}^n$ with a natural smoothness property with respect to the action of the affine group. This enables us to transfer several results established by Alesker for quasi-smooth valuations to the polytopal setting, and to reduce the problem to the translation invariant case of measurable valuations on polytopes.
\end{abstract}
\tableofcontents

\section{Introduction}

	Some of the most important problems in convex geometry deal with the properties and relations between the intrinsic volumes. These quantities admit several equivalent definitions in terms of various integral geometric formulas, most classically as the suitably normalized coefficient of the Steiner formula, which expresses the volume of tubular neighborhoods or radius $r$ as a polynomial in $r$ of degree bounded by the ambient space. In fact, this polynomial expansion holds for a much larger class of compact sets, and, in the case of smooth compact submanifolds, the resulting coefficients only depend on the induced metric and not the embedding of the submanifold into Euclidean space, as shown by Weyl \cite{WeylVolumeTubes1939}. In particular, these coefficients are indeed intrinsic invariants.\\
	Considered as functionals on the space $\mathcal{K}(\R^n)$ of convex bodies in $\R^n$, i.e. the set of all nonempty, compact, and convex subsets of $\R^n$ equipped with the Hausdorff metric, this universal behavior may be seen as a direct consequence of the following famous characterization result due to Hadwiger.
	\begin{theorem}[\cite{HadwigerVorlesungenuberInhalt1957}]
		\label{theorem:Hadwiger}
		Let $\mu:\mathcal{K}(\R^n)\rightarrow\R$ be a continuous, translation and $\SO(n)$-invariant valuation. Then $\mu$ is a linear combination of the intrinsic volumes $V_0,\dots,V_n$.
	\end{theorem}
	Here, a functional $\mu$ defined on some family of sets $\mathcal{S}$ is called a valuation if it satisfies
	\begin{align*}
		\mu(K)+\mu(L)=\mu(K\cup L)+\mu(K\cap L)
	\end{align*}
	for all $K,L\in\mathcal{S}$ such that $K\cup L,K\cap L\in\mathcal{S}$. This notion goes back to Dehn's solution of Hilbert's Third Problem, but has since become an important part of convex, integral, and differential geometry \cite{AleskerFaifmanConvexvaluationsinvariant2014,AleskerDescriptioncontinuousisometry1999,BernigEtAlHardLefschetztheorem2024,BoeroeczkyLudwigMinkowskivaluationslattice2019,FaifmanHofstaetterConvexvaluationsWhitney2025,FreyerEtAlUnimodularvaluationsEhrhart2025,LudwigEllipsoidsmatrixvalued2003,LudwigReitznerclassification$SLn$invariant2010,SchusterWannerer$GLn$contravariantMinkowski2012,SchusterWannererMinkowskivaluationsgeneralized2018,KnoerrSmoothvaluationsconvex2025,SchneiderSimplevaluationsconvex1996}. Moreover, there now exists a well developed theory of valuations on manifolds, \cite{AleskerTheoryvaluationsmanifolds.2006,BernigBroeckerValuationsmanifoldsRumin2007,BernigEtAlIntegralgeometrycomplex2014,BernigEtAlWeyltubetheorem2026}, which relies on a synthetic notion of smooth valuations that builds on the description of continuous and translation invariant valuations by Alesker \cite{AleskerDescriptiontranslationinvariant2001}.\\
	It has been a long standing open question in geometric valuation theory whether a characterization similar to \autoref{theorem:Hadwiger} holds for valuations on suitable families of sets in other space forms - more precisely, for convex bodies or polytopes on the sphere (see \cite{McMullenSchneiderValuationsconvexbodies1983}*{Problem~15.5}) or hyperbolic space. However, since all known proofs of Hadwiger's Theorem (see \cite{HadwigerVorlesungenuberInhalt1957,Chensimplifiedelementaryproof2004,KlainshortproofHadwigers1995,KnoerrRigidmotioninvariant2026}) make essential use (either explicitly or implicitly) of the special polynomial behavior of translation invariant valuations on polytopes with respect to Minkowski addition (see \cite{McMullenValuationsEulertype1977,PukhlikovKhovanskiiFinitelyadditivemeasures1992} for details), the arguments do not generalize easily. The only exceptions are the two dimensional cases, which were considered by Klain \cite{KlainIsometryinvariantvaluations2006} for the hyperbolic plane and Klain--Rota \cite{KlainRotaIntroductiongeometricprobability1997}*{Theorem 11.3.1} for the two dimensional sphere.\\
	In this article, we provide an affirmative solution to this problem for the spheres $\S^n$, $n\ge 1$. Let $\calP(\S^n)$ denote the space of spherical polytopes, that is, all subsets of $\S^n$ that are contained in an open hemisphere and that can be expressed as a finite intersection of closed hemispheres. We equip $\calP(\S^n)$ with the spherical Hausdorff metric. The main result of this article is the following characterization of the spherical intrinsic volumes, which we also denote by $V_0,\dots,V_n$.
	\begin{maintheorem}\label{maintheorem:HadwigerSphere}
		Let $\mu:\calP(\S^n)\rightarrow \R$ be a continuous and $\SO(n+1)$-invariant valuation. Then $\mu$ is a linear combination of the spherical intrinsic volumes $V_0,\dots,V_n$.
	\end{maintheorem}
	Note that \autoref{maintheorem:HadwigerSphere} directly implies the corresponding statement for continuous valuations on spherical convex bodies, since any such body may be approximated by spherical polytopes.\\

	As in all known proofs of \autoref{theorem:Hadwiger}, \autoref{maintheorem:HadwigerSphere} can be reduced to the following characterization of \emph{simple} valuations, where we call $\mu:\calP(\S^n)\rightarrow\R$ simple if $\mu$ vanishes on lower dimensional spherical polytopes.
	\begin{maintheorem}\label{maintheorem:HadwigerSphereSimple}
		Let $\mu:\calP(\S^n)\rightarrow \R$ be a continuous and $\SO(n+1)$-invariant valuation. If $\mu$ is simple, then $\mu$ is a multiple of the spherical Lebesgue measure.
	\end{maintheorem}

	Let us discuss the strategy behind our proof. Although the focus on polytopes suggests a more discrete approach, the main idea is rooted in Alesker's notion of \emph{quasi-smooth} valuations on convex bodies \cite{AleskerTheoryvaluationsmanifolds.2006}. Informally speaking, these are valuations $\mu:\K(\R^n)\rightarrow\R$ such that the maps
	\begin{align}
		\label{eq:ScalingMapConvexBody}
		t\mapsto \mu(tK+x), \quad K\in\mathcal{K}(\R^n),x\in\R^n,
	\end{align}
	satisfy a rather strong uniform differentiability condition in $t=0$. In particular, this property is not satisfied by a general continuous valuation and it is not known whether every continuous valuation  on $\K(\R^n)$ can be approximated in a suitable sense by valuations of this type. The motivation to consider this class of valuations comes from the observation that their derivatives in $t=0$ can be considered as translation invariant valuations on the tangent space at the point $x\in\R^n$ (under suitable vanishing properties of lower order derivatives, see \autoref{section:differentiabilityProperties} below), so these valuations can be studied using results from the translation invariant case.\\
	
	The first ingredient in our proof of \autoref{maintheorem:HadwigerSphere} is the observation that a weak version of this differentiability property holds automatically for valuations on the space $\calP(\R ^n)$ of polytopes in $\R^n$ that are \emph{affine smooth} in the follow sense: If $\Aff(n,\R)$ denotes the affine group, then for every $P\in\calP(\R ^n)$, 
	\begin{align*}
		A\mapsto \mu(A(P)), \quad A\in\Aff(n,\R),
	\end{align*}
	defines a smooth function on $\Aff(n,\R)$. Under these assumptions, the corresponding map in Eq.~\eqref{eq:ScalingMapConvexBody} is pointwise (i.e. for every fixed $P\in \calP(\R^n)$) differentiable in $t=0$, compare \autoref{theorem:DifferentiabilityAffineSmooth}. In essence, this remarkable property is a consequence of the Canonical Simplex Decomposition (see \autoref{section:polytopes}) and the Inclusion-Exclusion Principle, which lets us replace the derivative in $t=0$ by suitable derivatives of the action of $\Aff(n,\R)$ on different polytopes. This enables us to transfer some of Alesker's constructions from \cite{AleskerTheoryvaluationsmanifolds.2006} to affine smooth valuations on polytopes in $\R^n$.\\
	
	In order to apply this to valuations on $\S^n$, it is useful to identify $P\in\calP(\S^n)$ with the polyhedral cone
	\begin{align*}
		\hat{P}:=\{tu\in \R^n: u\in P, t\ge 0\}.
	\end{align*}
	Note that $P=\hat{P}\cap \S^n$, so there is a $1$-to-$1$-correspondence between polyhedral cones contained (except for the origin) in an open half space and spherical polytopes. In this interpretation, there is an obvious action of the general linear group $\GL(n+1,\R)$ on $\S^n$ and $\calP(\S^n)$ corresponding to the action of $\GL(n+1,\R)$ on the ambient linear space, compare \autoref{section:sphericalPolytopes}, and a standard convolution argument reduces the proof of \autoref{maintheorem:HadwigerSphereSimple} to a corresponding result for $\GL(n+1,\R)$-smooth valuations in the sense of \autoref{section:glSmooth}. Moreover, this action satisfies some strong compatibility properties with the gnomonic projection, which implies that, informally, any $\GL(n+1,\R)$-smooth valuation is affine smooth in coordinate charts corresponding to the gnomonic projection. \autoref{maintheorem:HadwigerSphereSimple} may then be reduced to showing that a simple and $\SO(n+1)$-invariant valuation on $\S^n$ satisfying these smoothness properties restricts to smooth measures in coordinate charts. This last property can equivalently be encoded in the vanishing of the derivatives of Eq.~\eqref{eq:ScalingMapConvexBody} in $t=0$ up to order $n-1$, compare \autoref{section:differentiabilityProperties}.\\
	
	This is where the last ingredient for the proof enters. It was shown in \cite{KnoerrRigidmotioninvariant2026} that Hadwiger's characterization in \autoref{theorem:Hadwiger} holds under weaker regularity requirements for translation invariant valuations on polytopes: It already holds for rigid motion invariant \emph{measurable} valuations. In particular, this result can in principle be applied to the valuations obtained by differentiating  Eq.~\eqref{eq:ScalingMapConvexBody}, which we use to establish \autoref{maintheorem:HadwigerSphereSimple} for $\GL(n+1)$-smooth valuations. The main challenge is a lack of a proper chain rule that lets us compare the derivatives corresponding to Eq.~\eqref{eq:ScalingMapConvexBody} in different coordinate charts. However, since we are only interested in showing that the relevant derivatives vanish, this can be circumvented by using suitable estimates derived from Taylor's formula with remainder.
	
	\begin{acknowledgement}
		This research was funded in whole or in part by the Austrian Science Fund (FWF), \href{https://www.doi.org/10.55776/PAT4205224}{10.55776/PAT4205224}. 
	\end{acknowledgement} 
\section{Preliminaries}

	\subsection{Polytopes and convex bodies in $\R^n$}
		\label{section:polytopes}
		We refer to \cite{SchneiderConvexbodiesBrunn2014} for a general background on convex bodies.\\	Recall that $\K(\R^n)$ denotes the space of convex bodies equipped with the Hausdorff metric. We denote the subspace of polytopes in $\R^n$ by $\calP(\R^n)\subset \K(\R ^n)$ and equip it with the induced metric. In particular, we call a map $\mu:\calP(\R ^n)\rightarrow\R$ measurable if the preimage of any open subset of $\R^n$ is a Borel set of $\calP(\R ^n)$. For an affine subspace $E\subset\R^n$, we will also write $\calP(E)\subset\calP(\R^n)$ for the subset of polytopes contained in $E$.\\
		We require the following description of convergence in $\mathcal{K}(\R ^n)$.
		\begin{theorem}[\cite{SchneiderConvexbodiesBrunn2014}*{Theorem 1.8.8}]
			\label{theorem:ConvergenceHausdorff}
			Let $(K_j)_j$ be a sequence in $\mathcal{K}(\R^n)$ and $K\in\mathcal{K}(\R ^n)$. Then $(K_j)_j$ converges to $K$ if and only if
			\begin{enumerate}
				\item each point in $K$ is the limit of a sequence $(x_j)_j$ with $x_j\in K_j$;
				\item the limit of every convergent sequence $(x_{i_j})_j$ with $x_{i_j}\in K_{i_j}$ belongs to $K$.
			\end{enumerate}
		\end{theorem}
		
		We note the following consequence for sequences of simplices.
		\begin{corollary}\label{corollary:ConvergenceSimplices}
			Let $(\Delta_j)_j$ be a bounded sequence of full dimensional simplices in $\R^n$ and assume that $A\subset \R^n$ is a compact subset with nonempty interior that satisfies the two condition in \autoref{theorem:ConvergenceHausdorff} for the sequence $(\Delta_j)_j$. Then $A$ is a full dimensional simplex and $(\Delta_j)_j$ converges to $A$ in the Hausdorff metric. 
		\end{corollary}
		\begin{proof}
			By the Blaschke Selection Theorem (see \cite{SchneiderConvexbodiesBrunn2014}*{Theorem 1.8.7}), the sequence $(\Delta_j)_j$ has a convergent subsequence, say $(\Delta_{i_j})_{j}$. Since the set of all polytopes with at most $n+1$ vertices is closed in the Hausdorff metric, the limit $P$ is a polytope with at most $n+1$ vertices. Now note that both $P$ and $A$ satisfy the conditions in \autoref{theorem:ConvergenceHausdorff} for the sequence $(\Delta_{i_j})_{j}$, so we have $P\subset A$ and $A\subset P$. Thus we see that $P=A$, so $A$ is the unique limit point of $(\Delta_j)_j$, which shows that the sequence converges to $A$. Since $A$ has nonempty interior, $P=A$ is a full dimensional polytope with $(n+1)$-vertices and therefore a simplex.
		\end{proof}

		If $P\in \calP(\R^n)$ is a finite union of polytopes $P_j$ such that $P_i\cap P_j$ is lower dimensional for all $1\le i,j\le N$, we will write $P=\bigsqcup_{j=1}^N P_j$. Let $\Delta_j=\{x\in \R^j: 0\le x_1\le\dots\le x_j\le 1\}$ denote the $j$-dimensional standard simplex. The following is known as the Canonical Simplex Decomposition, compare \cite{HadwigerVorlesungenuberInhalt1957}*{Section 1.2.6}. We refer to \cite{AleskerIntroductiontheoryvaluations2018}*{Theorem~1.1.} and \cite{LudwigMussnigValuationsConvexBodies2023}*{Theorem~4.2} for proofs. 
		\begin{theorem}
			\label{theorem:SimplexDeomposition}
			For $s,t\ge0$,
			\begin{align*}
				(s+t)\Delta_n=\bigsqcup_{j=0}^n \left((s\Delta_j)\times(t\Delta_{n-j}+\underbrace{(s,\dots,s)}_{(n-j)-\text{times}})\right).
			\end{align*}
		\end{theorem}
		Note that the result is usually stated for $s,t>0$ since the polytopes on the right hand side may be lower dimensional otherwise and the decomposition becomes trivial.\\
		Assume for now that $s,t>0$. Note that for $x\in (s+t)\Delta_n$ there exists a unique $0\le k\le n$ such that
		\begin{align*}
			0\le x_1\le \dots\le x_k\le s<x_{k+1}\le \dots\le x_n\le s+t.
		\end{align*}
		Similarly, there exists a unique $0\le j\le n$ such that
		\begin{align*}
			0\le x_1\le \dots\le x_j< s\le x_{j+1}\le \dots\le x_n\le s+t.
		\end{align*}
		Note that this implies $j\le k$. 
		Then $x\in(s+t)\Delta$ belongs to $s\Delta_i\times[t\Delta_{n-i}+(s,\dots,s)]$ if and only if $j\le i$ and $i\le k$. In particular, for $0\le a<b\le n$, 
		\begin{align*}
			&\left(s\Delta_a\times[t\Delta_{n-a}+(s,\dots,s)]\right)\cap \left(s\Delta_b\times[t\Delta_{n-b}+(s,\dots,s)]\right)\\
			=&(s\Delta_a)\times \{\underbrace{(s,\dots,s)}_{(b-a)-\text{times}}\}\times[t\Delta_{n-b}+(s,\dots,s)].
		\end{align*}
		By continuity, the same relation holds for $s,t\ge 0$ (compare \cite{SchneiderConvexbodiesBrunn2014}*{Theorem 1.8.10}). In particular, we obtain the following.
		\begin{corollary}\label{corollary:SimplexDecompositionIntersections}
			For $s,t\ge 0$, let $P_j(s,t)=s\Delta_j\times[t\Delta_{n-j}+\underbrace{(s,\dots,s)}_{(n-j)-\text{times}}]$. For $\emptyset\ne I\subset\{0,\dots,n\}$, we have
			\begin{align*}
				\bigcap_{j\in I}P_j(s,t)=(s\Delta_{\min I})\times \underbrace{\{(s,\dots,s)\}}_{(\max I-\min I)-\text{times}}\times[t\Delta_{n-\max(I)}+(s,\dots,s)].
			\end{align*}
		\end{corollary}
	
	\subsection{Valuations on $\calP(\R^n)$}
		As a general background on valuations on polytopes and convex bodies in $\R^n$, we refer to \cite{SchneiderConvexbodiesBrunn2014}*{Section~6} and \cite{KlainRotaIntroductiongeometricprobability1997}. We will only require the following three results.\\
		First, valuations on polytopes satisfy the following more general additivity property, compare
		\cite{SchneiderConvexbodiesBrunn2014}*{Theorem~6.2.3}.
		\begin{theorem}\label{theorem:InclusionExclusionPrinciple}
			Every valuation $\mu:\calP(\R^n)\rightarrow \R$ satisfies the Inclusion-Exclusion Principle: If $P_i\in \calP(\R^n)$, $1\le i\le N$, are polytopes such that $\bigcup_{i=1}^N P_i\in \calP(\R^n)$, then
			\begin{align*}
				\mu\left(\bigcup_{i=1}^NP_i\right)=\sum_{r=1}^N(-1)^{r-1}\sum_{1< i_1\le \dots< i_r\le N}\mu\left(\bigcap_{j=1}^rP_{i_j}\right).
			\end{align*}
			Here, we set $\mu(\emptyset):=0$.
		\end{theorem}	
		The following result is due to Hadwiger \cite{HadwigerVorlesungenuberInhalt1957} (see also \cite{SchneiderConvexbodiesBrunn2014}*{Theorem~6.4.3}).
		\begin{proposition}\label{proposition:VolumeCharacterizationHomMeasurable}
			Every translation invariant and $n$-homogeneous valuation on $\calP(\R ^n)$ is a multiple of the Lebesgue measure.
		\end{proposition}	
		
		Our proof of \autoref{maintheorem:HadwigerSphereSimple} relies on the polytopal version of Hadwiger's characterization of the intrinsic volumes from \cite{KnoerrRigidmotioninvariant2026}. More specifically, we require the following vanishing result.
		\begin{proposition}[\cite{KnoerrRigidmotioninvariant2026}*{Proposition~5.4}]
			\label{proposition:simpleValuationsVanish}
			Let $0\le k\le n-1$ and $\mu:\calP(\R^n)\rightarrow\R$ be a measurable, translation and $\SO(n)$-invariant, as well as simple valuation that is homogeneous of degree $k$. Then $\mu=0$.
		\end{proposition}
		
	\subsection{Spherical polytopes and the gnomonic projection}
		\label{section:sphericalPolytopes}
		We refer to \cite{KlainRotaIntroductiongeometricprobability1997}*{Section 11} for a discussion of spherical convexity.\\		
		The spherical distance on the unit sphere $\S^n$ is given by $d(u,v)=\arccos(\langle u,v\rangle)$ for $u,v\in\S^n$. If we denote the set of points of distance at most $r>0$ from a set $A$ by $A_r$, then the spherical Hausdorff distance between two closed subsets $A,B\subset\S^n$ is therefore given by
		\begin{align*}
			\delta(A,B)=\inf\{r>0: A\subset B_r, B\subset A_r\}.
		\end{align*}
		For $u\in \S^n$, let $\S^n_+(u):=\{v\in \S^n: \langle u,v\rangle >0\}$ denote the open hemisphere centered at $u$. We consider the gnomonic projection
		\begin{align*}
			G_u:\S^n_+(u)&\rightarrow u^\perp\\
			v&\mapsto \frac{v-\langle u,v\rangle u}{\langle u,v\rangle},
		\end{align*}
		which is a well defined diffeomorphism with inverse $G_u^{-1}(x)=\frac{u+x}{|u+x|}$. Geometrically, the gnomonic projection maps the point $v\in \S^n_+(u)$ to the intersection of the corresponding ray with the affine hyperplane $u+u^\perp$. In particular, for a spherical polytope $P$ contained in $\S^n_+(u)$, its image is obtained by considering the intersection of the convex polyhedral cone 
		\begin{align*}
			\hat{P}=\{tv:v\in P,t\ge0\}\subset \R^{n+1}
		\end{align*}
		with $u+u^\perp$. If we denote the space of spherical polytopes contained in $\S ^n_+(u)$ by $\calP(\S^n_+(u))$, this implies the following (see also \cite{BesauSchusterBinaryoperationsspherical2016}*{Corollary  4.5}).
		\begin{lemma}\label{lemma:gnomonicHomeomorphismPolytopes}
			For every $u\in \S^n$, the gnomonic projection $G_u:\S^n_+(u)\rightarrow u^\perp$ induces a homeomorphism between $\calP(\S^n_+(u))$ and $\calP(u^\perp)$.
		\end{lemma}
		\begin{remark}
			Note in particular that $G_u$ maps full dimensional simplices to full dimensional simplices, where a simplex in $\S^n$ is by definition the spherical convex hull of $n+1$ linearly independent points, or equivalently, the intersection of $\S^n$ with a full dimensional simplicial and pointed cone in $\R^{n+1}$.
		\end{remark}
		
		The interpretation of spherical polytopes as the intersection of polyhedral cones with $\S^n$ gives rise to the following action of $\GL(n+1,\R)$ on $\S^n$: For $g\in\GL(n,\R)$, we define
		\begin{align*}
			g(v):=\frac{gv}{|gv|},\quad\text{for}~v\in \S^n,
		\end{align*}
		where we have the usual action on vectors in $\R^{n+1}$ on the right hand side. It is easy to check that this defines a continuous action of $\GL(n+1,\R)$. Moreover, since it corresponds to the usual action of $\GL(n+1,\R)$ on $\R^{n+1}$, which preserves polyhedral cones, this directly implies the following.
		\begin{lemma}\label{lemma:groupActionSphericalPolytopesWellDefined}
			The action of $\GL(n+1,\R)$ on $\S^n$ induces a well defined and continuous action on $\calP(\S^n)$.
		\end{lemma}
		Implicit in the previous result is that the polytope $g(P)$ is again properly contained in an open hemisphere, which is the case since otherwise the corresponding cone in $\R^{n+1}$ would contain an at least $1$-dimensional subspace, so the same would hold for the cone corresponding to $P$.\\
		
		We will be interested in the interaction of this group action with the gnomonic projection. Consider the action of the affine group on $u^\perp$. For $A\in\GL(u^\perp)$ and $x_0\in u^\perp$, define $g^{u,x_0}_{A}\in\GL(n+1,\R)$ by
		\begin{align*}
			g^{u,x_0}_{A}|_{u^\perp}=&A,\\
			g^{u,x_0}_{A} u=& u+x_0.
		\end{align*}
		Note that this uniquely defines $g^{u,x_0}_{A}$ since $\R^{n+1}=\R u\oplus u^\perp$. We have the following simple relation.
		\begin{lemma}\label{lemma:EquivarianceAction}
			Fix $u\in\S^n$, $A\in\GL(u^\perp)$ and $x_0\in u^\perp$. Then for $x\in u^\perp$,
			\begin{align*}
				G^{-1}_u(Ax+x_0)=g^{u,x_0}_{A} (G_u^{-1}(x)).
			\end{align*}
		\end{lemma}
		We will abuse notation slightly and write $g^{u_0,x_0}_t\in \GL(n+1,\R)$ for the  the unique linear map satisfying 
		\begin{align*}
			g^{u,x_0}_t|_{u^\perp}=&t Id_{u^\perp},\\
			g^{u,x_0}_tu=&u+x_0.
		\end{align*} 
		Note that these maps satisfy the following compatibility condition.
		\begin{corollary}\label{corollary:ScalingMapsTransition}
			For $u\in \S^n$, $x\in u^\perp$, and $V\in \SO(n+1)$,
			\begin{align*}
				Vg^{u,x}_tV^{-1}=g^{Vu,Vx}_t.
			\end{align*}
		\end{corollary}

		The following result is one of the key ingredients for the proofs of our main results.
		\begin{lemma}\label{lemma:LimitCoordinateChangeSimplex}
			Let $\Delta_0\in\calP(\S^n_+(u_0))$ be a full dimensional simplex with vertex $u_0\in \Delta_0$. For every $v\in \calP(\S^n_+(u_0))$, the simplices defined for $t>0$ small enough by
			\begin{align*}
				\Delta_t:=\frac{1}{t}\left[G_{v} (g_t^{u_0,G_{u_0}(v)}(\Delta_0))\right]\in \calP(v^\perp)
			\end{align*}
			have a vertex at the origin in $v ^\perp$ and converge for $t\rightarrow0$ to a full dimensional simplex $\Delta$ in $v^\perp$ with vertex $0\in \Delta$.
		\end{lemma}
		\begin{proof}
			Note that $v_0=g_t^{u_0,G_{u_0}(v)}(u_0)$ is a vertex of $g_t^{u_0,G_{u_0}(v)}(\Delta_0)$, so these simplices have $0$ as a vertex by construction.	Fix $t_0>0$ small enough such that the simplex $\Delta_t$ is well defined for all $t\in (0,t_0)$. Then the map
			\begin{align*}
				(t,y)\mapsto G_{v} (g_t^{u_0,G_{u_0}(v)}(y))
			\end{align*}
			is well defined on a neighborhood of $\Delta_0\subset \S^n_+(u_0)$ for all $t\in[0,t_0)$ and satisfies
			\begin{align*}
				\lim_{t\rightarrow0}\frac{G_{v} (g_t^{u_0,G_{u_0}(v)}(y))}{t}=\frac{d}{dt}\Big|_0G_{v} (g_t^{u_0,G_{u_0}(v)}(y)).
			\end{align*}
			If we write $y=(y-\langle y,u_0\rangle u_0)+\langle y,u_0\rangle u_0$, we obtain
			\begin{align*}
				G_v\left(g_t^{u_0,G_{u_0}(v)}(y)\right)=&G_v\left(\frac{t(y-\langle y,u_0\rangle u_0)+\langle y,u_0\rangle [u_0+G_{u_0}(v)]}{|t(y-\langle y,u_0\rangle u_0)+\langle y,u_0\rangle [u_0+G_{u_0}(v)]|}\right)\\
				=&\frac{f(t)-\langle f(t),v\rangle v}{\langle f(t),v\rangle}
			\end{align*}
			for $f(t)=tG_{u_0}(y)+\langle v,u_0\rangle^{-1}v$, where we used $u_0+G_{u_0}(v)=\langle v,u_0\rangle^{-1}v$, so
			\begin{align*}
				\frac{d}{dt}\Big|_0G_{v} (g_t^{u_0,G_{u_0}(v)}(y))=&\frac{f'(0)\langle f(0),v\rangle-f(0)\langle f'(0),v\rangle}{\langle f(0),v\rangle^2}\\
				=&\langle v,u_0\rangle\left[G_{u_0}(y)- \langle G_{u_0}(y),v\rangle v\right]=:F(y).
			\end{align*}
			Note that this is well defined for all $y\in \S^n_+(u_0)$ and thus defines a smooth function $F:\S^n_+(u_0)\rightarrow v^\perp$. More explicitly, 
			\begin{align*}
				F(y)=&\frac{\langle v,u_0\rangle}{\langle y,u_0\rangle}[y-\langle y,u_0\rangle u_0- \langle y,v\rangle v+\langle y,u_0\rangle \langle u_0,v\rangle v].
			\end{align*}
			Thus, if $\gamma(t)$ is a curve with $\gamma(0)=u_0$ and $\gamma'(0)=x\in u_0^\perp$, then 
			\begin{align*}
				\frac{d}{dt}\Big|_0F\left(\gamma(t)\right)=\langle v,u_0\rangle[x-\langle x,v\rangle v]&.
			\end{align*}
			Note that this does not vanish for $x\ne 0$: Either $x\perp v$ or $\langle x,v\rangle\ne0$, and then $\langle x-\langle x,v\rangle v,u_0\rangle=-\langle x,v\rangle \langle  v,u_0\rangle\ne0$ since both factors do not vanish. Thus the differential of $F$ is invertible in $u_0$. Since $u_0\in \Delta_0$, this shows that
			\begin{align*}
				\Delta:=F(\Delta_0)\subset v^\perp
			\end{align*}
			has nonempty interior. Since $(t,y)\mapsto \frac{1}{t}G_{v} (g_t^{u_0,G_{u_0}(v)}(y))$ extends to a continuous function to $t=0$ (with limit $F(y)$), \autoref{corollary:ConvergenceSimplices} shows that this set is a full dimensional simplex with $\lim\limits_{t\rightarrow0}\Delta_t=\Delta$. In particular, $0$ is a vertex of $\Delta$.
		\end{proof}

\section{Affine smooth valuations on $\R^n$}
	\subsection{Differentiability properties of affine smooth valuations}
	\label{section:differentiabilityProperties}
		\begin{theorem}\label{theorem:DifferentiabilityAffineSmooth}
		Let $\mu:\calP(\R^n)\rightarrow\R$ be a valuation and assume that the map
		\begin{align*}
			\Aff(n,\R)&\rightarrow \R\\
			A&\mapsto \mu(A(P))
		\end{align*} is smooth for every fixed $P\in\calP(\R^n)$. Then the map
		\begin{align*}
			%\label{eq:smoothFunctionScaling}
			\begin{split}[0,\infty)\times\GL(n,\R)\times  \R^n&\rightarrow\R\\
				(t,g,x)&\mapsto \mu(g[tP+x])
			\end{split}
		\end{align*}
		is smooth for every fixed $P\in\calP(\R^n)$.
	\end{theorem}
	\begin{proof}
		We will show the claim by induction on the dimension $n$, where the case $n=0$ is trivial. Thus assume that the claim holds for all such valuations on an at most $n-1$ dimensional real vector space. Given a valuation $\mu$ on $\R^n$ with these properties, the induction assumption implies that the claim holds for all polytopes of dimension at most $n-1$.\\
		
		We will now show that $\mu$ satisfies the following property:
		If $\R^n=E\oplus F$ is a direct sum decomposition, then for every $P_E\in \calP(E)$ and $P_F\in \calP(F)$, the map
		\begin{align*}
			[0,1)\times \GL(n,\R)\times \GL(F,\R)\times\R^n&\rightarrow\R\\
			(t,g,A,x)&\mapsto \mu(g[tP_E+AP_F+x]) 
		\end{align*}
		is smooth. We will use induction on $j=\dim E$, where the case $j=0$, i.e. $F=\R^n$, holds by our assumptions on $\mu$.\\
		Thus assume that the claim holds for all such direct sum decompositions $\R^n=E'\oplus F'$ with $\dim E'\le j-1$. Using the Inclusion-Exclusion Principle, it is sufficient to prove the claim when $P_E$ is a full dimensional simplex in $E$. Changing coordinates if necessary, we may assume that we are given the standard simplex $\Delta$ in $E\cong\R^j$. Using the Canonical Simplex Decomposition in \autoref{theorem:SimplexDeomposition}, we write	for $0\le t<1$,
		\begin{align*}
			\Delta= \bigsqcup_{j=0}^{j} \left[t\Delta_i \times[(1-t)\Delta_{j-i}]+t\sum_{l=i+1}^{j}e_l\right],
		\end{align*}
		where $e_i$, $1\le i\le j$, denotes the standard basis of $E\cong\R^j$. Thus, the Inclusion-Exclusion Principle from \autoref{theorem:InclusionExclusionPrinciple} implies
		\begin{align*}
			&\mu(g[\Delta+AP_F+x])\\
			=&\mu\left(g\left[t\Delta+AP_F+x\right]\right)+\sum_{l=1}^{j+1}(-1)^{l-1}\sum_{\substack{I\subset \{0,\dots,j\},|I|=l,\\ I\ne \{j\}}}\mu\left(g\left[\bigcap_{i\in I}P_i(t)+AP_F+x\right]\right)
		\end{align*}
		for $P_i(t)=t\Delta_i\times [(1-t)\Delta_{j-i}]+t\sum_{l=i+1}^je_l$. Due to \autoref{corollary:SimplexDecompositionIntersections}, we have
		\begin{align*}
			\bigcap_{i\in I}P_i(t)=t\Delta_{\min I}\times\{(0,\dots,0)\}\times [(1-t)\Delta_{j-\max I}]+t\sum_{l=\min I+1}^{j}e_l.
		\end{align*}	
		Thus 
		\begin{align*}
			&\mu\left(g\left[t\Delta+AP_F+x\right]\right)\\
			=&	\mu(g\left[\Delta+AP_F+x\right])\\
			&-\sum_{l=1}^{j+1}(-1)^{l-1}\sum_{\substack{I\subset \{0,\dots,j\},|I|=l,\\ I\ne \{j\}}}\mu\left(g\left[t\Delta_{\min I}+A_t\tilde{P}_I+t\sum_{l=\min I+1}^{j}e_l+x\right]\right)
		\end{align*}
		for the polytopes $\tilde{P}_I=\{(0,\dots,0)\}\times \Delta_{j-\max I} +P_F$ in \begin{align*}
			\tilde{F}_I:=\mathrm{span}(e_{\min I+1},\dots e_j)\oplus F,
		\end{align*} and the matrix $A_t\in\GL(\tilde{F}_I)$ acting as $(1-t)$ on the first summand and as $A$ on the second. Note that the map $[0,1)\times \GL(F)\mapsto \GL(\tilde{F}_I)$, $(t,A)\mapsto A_t$ is well defined and smooth. Since $\Delta_{\min I}$ is contained in a subspace of dimension strictly less than $j$ for every index set $I$ in the sum, we may apply the induction assumption to see that the right hand side defines a smooth function on $[0,1)\times\GL(n)\times \GL(F)\times\R^n$. Thus, the same holds for the left hand side, which completes the induction.\\
		Now note that our original claim follows from the case $j=n$.
	\end{proof}
	\begin{remark}
		Note that the previous result does not require that $\mu$ is continuous as a functional on $\calP(\R^n)$. We will nevertheless only consider continuous valuations for the remaining sections of this article, since a key result requires that the pointwise derivative of this expression defines a measurable valuation (compare \autoref{proposition:PropertiesLambdaDifferential} below).
	\end{remark}

	\subsection{Filtration}
		\label{section:AffineSmoothValuations}
		Let $\CV(\R^n)$ denote the space of all continuous valuations $\mu:\calP(\R^n)\rightarrow\R$. We denote by $\CV(\R^n)^{sm}$ the subspace of all valuations $\mu\in\CV(\R ^n)$ such that the map
		\begin{align*}
			\Aff(n,\R)&\rightarrow\R\\\
			g&\mapsto \mu(g(P))
		\end{align*}
		is smooth for every fixed $P\in \calP(\R^n)$. We will call these valuations \emph{affine smooth} for brevity.
		\begin{remark}
			Note that this is a pointwise condition. For some more advanced constructions, it might be more natural to consider the smooth vectors of the natural representation of $\Aff(n,\R)$ on $\CV(\R^n)$ equipped with the compact-open topology instead (which would be more closely related to the notion used in \cite{AleskerTheoryvaluationsmanifolds.2006}), however, since the pointwise differentiability is sufficient for our purposes, we will use this weaker notion. 
		\end{remark}
		Similar to \cite{AleskerTheoryvaluationsmanifolds.2006}, we may consider the following subspaces, which are well-defined due to \autoref{theorem:DifferentiabilityAffineSmooth}.
		\begin{definition}
			For positive $k\in\mathbb{N}$ we define
			\begin{align*}
				&W_k(\R^n):=\\
				&\left\{\mu\in \CV(\R^n)^{sm}:\frac{d^{i}}{dt^{i}}\Big|_0\mu(tP+x)=0~\text{for}~P\in\calP(\R^n), x\in\R^n, 0\le i<k\right\}.
			\end{align*}
			For $k=0$, we set $W_0(\R^n)=\CV(\R^n)^{sm}$ for completeness.
		\end{definition}
		Note that this defines a decreasing filtration $W_0(\R^n)\supset W_1(\R^n)\supset...$ on $\CV(\R^n)^{sm}$. The motivation behind this construction is that it allows us to relate affine smooth valuations to translation invariant valuations.
		\begin{proposition}\label{proposition:PropertiesLambdaDifferential}
			For $\mu\in W_k(\R^n)$ and $x\in\R^n$, the map $D_x^k\mu:\calP(\R^n)\rightarrow\R$ defined by
			\begin{align*}
				D_x^k\mu(P):=\frac{d^k}{dt^k}\Big|_0\mu(tP+x)\quad\text{for}~P\in\calP(\R^n)
			\end{align*}
			is a translation invariant and measurable valuation that is homogeneous of degree $k$.
		\end{proposition}
		\begin{proof}
			This is clear for $k=0$, since $D_x^0\mu(P)=\mu(x)$. Thus let $k>0$.
			First note that since all lower order derivatives vanish, $D_x^k\mu:\calP(\R^n)\rightarrow\R$ may also be calculated by
			\begin{align*}
				D_x^k\mu(P)=k!\lim_{t\rightarrow 0^+}\frac{\mu(tP+x)}{t^k}.
			\end{align*}
			In particular, it is measurable as the pointwise limit of a sequence of continuous valuations. Obviously, this limit has the valuation property. In order to see that it is translation invariant, note that
			\begin{align*}
				\mu(t(P+y)+x)=\mu(tP+(x+sy))|_{s=t},
			\end{align*}
			so a repeated application of the chain rule implies the result. Similarly, the fact that $D_x^k\mu$ is $k$-homogeneous is a direct consequence of the chain rule.
		\end{proof}
	\begin{remark}\label{remark:VanishingDerivativeFiltration}
		Note that, by definition, $\mu\in W_k(\R^n)$ belongs to $W_{k+1}(\R^n)$ if and only if $D^k_x\mu=0$ for all $x\in\R^n$.
	\end{remark}
		
	The proofs of the following two results are essentially identical to the proofs of the corresponding results in \cite{AleskerTheoryvaluationsmanifolds.2006}, however, since they require some minor but essential modifications due to our weaker regularity assumptions, we include the complete arguments. 
	\begin{proposition}\label{proposition:FiltrationTerminates}
		$W_{n+1}(\R^n)=0$.
	\end{proposition}
	\begin{proof}
		The argument is taken almost verbatim from \cite{AleskerTheoryvaluationsmanifolds.2006}*{Proposition 3.1.1}. We will use induction on the dimension, where the case $n=0$ is trivial. Thus assume that the claim holds for all real vector spaces of dimension at most $n-1$.\\
		Given $\mu\in W_{n+1}(\R^n)$, the induction assumption thus implies that $\mu$ is a simple valuation. Since we may triangulate any given polytope, it is thus sufficient to show that $\mu(\Delta)=0$ for any simplex $\Delta$. Changing coordinates if necessary, we may assume that $\Delta=\{x\in\R^n:0\le x_1\le\dots\le x_n\le1\}$ is the standard simplex. For $0\le i\le j\le n$, consider the sets $T_{i,j}$ given by
		\begin{align*}
			\{(x_1,\dots,x_n):0\le x_i\le x_{i+1}\le\dots\le x_j\le 1~\text{and}~x_l=0~\text{for}~l<i~\text{and}~l>j\}.
		\end{align*}
		Given a sequence $0<j_1<\dots<j_{l-1}<n$, we consider
		\begin{align*}
			T_{j_1\dots j_{l-1}}:= T_{0j_0}+\dots+T_{j_{l-1}n}.
		\end{align*}
		Note that any point $x\in \Delta$ satisfies
		\begin{align}
			\label{eq:defEqMultiindex}
			0\le x_1=\dots=x_{j_1}<x_{j_1+1}=\dots=x_{j_2}<\dots<x_{j_{l-1}+1}=\dots x_{n}\le 1
		\end{align}
		for a unique such sequence $0<j_1<\dots<j_{l-1}<n$ and $1\le l\le n-1$. If we set for $N\in\mathbb{N}$ and a given sequence $0<j_1<\dots<j_{l-1}<n$
		\begin{align*}
			R_{j_1\dots j_{l-1}}(N)=\left\{z\in \frac{1}{N}\mathbb{Z}^n\cap \Delta: z~\text{satisfies}~\eqref{eq:defEqMultiindex}\right\},
		\end{align*}
		then we have the decomposition
		\begin{align*}
			\Delta=\bigsqcup_{\substack{1\le l\le n-1,\\ 0<j_1<\dots<j_{l-1}<n}}\bigsqcup_{z\in R_{j_1\dots j_{l-1}}(N) }\left(\frac{1}{N}T_{j_1\dots j_{l-1}}+z\right).
		\end{align*}
		Since $\mu$ is simple, this implies
		\begin{align}
			\label{eq:multisum}
			\mu(\Delta)=\sum_{l=1}^{n-1}\sum_{0<j_1<\dots<j_{l-1}<n}\sum_{z\in R_{j_1\dots j_{l-1}}(N)}\mu\left(\frac{1}{N}T_{j_1\dots j_{l-1}}+z\right).
		\end{align}
		Note that we are only considering translated and rescaled copies of the finite family  of polytopes $T_{j_1\dots j_{l-1}}$. Recall that the function $(t,x)\mapsto \mu(tP+x)$ is smooth on $[0,\infty)\times\R^n$ for every fixed polytope $P$ by \autoref{theorem:DifferentiabilityAffineSmooth}. Since $\mu\in W_{n+1}(\R^n)$, all derivatives in $t=0$ vanish up to order $n$. For a fixed polytope $P\in\calP(\R^n)$ and $R>0$, we may thus apply Taylor's formula with remainder to obtain for every $\epsilon>0$ an index $N(P,\epsilon,R)$ such that
		\begin{align*}
			\left|\mu\left(\frac{1}{N}P+x\right)\right|\le \epsilon N^{-n}
		\end{align*}
		for all $N\ge N(P,\epsilon,R)$ and $|x|\le R$. Now note that the sets $R_{j_1\dots j_{l-1}}(N)$ are bounded independent of $N$. Since we only have a finite family of polytopes, we thus find for $\epsilon>0$ an index $N(\epsilon)$ such that for all $1\le l\le n-1$ and sequences $0< j_1<\dots<j_{l-1}<n$,
		\begin{align*}
			\left|\mu\left(\frac{1}{N}T_{j_1\dots j_{l-1}}+z\right)\right|\le \epsilon N^{-n}
		\end{align*}
		for all $N\ge N(\epsilon)$  and $z\in R_{j_1\dots j_{l-1}}(N)$. Since $\left|\frac{1}{N}\mathbb{Z}^n\cap\Delta\right|\le N^n$, we may estimate the individual terms in the sum in Eq.~\eqref{eq:multisum} for $N\ge N(\epsilon)$, and obtain
		\begin{align*}
			|\mu(\Delta)|\le n\cdot C\cdot N^n\cdot \epsilon N^{-n}=nC\epsilon
		\end{align*}
		where $C$ is a bound on the number of sequences satisfying $0< j_1<\dots<j_{l-1}<n$. Since the constant $nC$ is independent of $\epsilon>0$, we obtain $\mu(\Delta)=0$, which shows the claim.
	\end{proof}
	
	\begin{proposition}\label{proposition:classificationWnDensity}
		If $\mu\in W_n(\R^n)$, then $\mu$ is a smooth measure.
	\end{proposition}
	\begin{proof}
		The argument is essentially identical to \cite{AleskerTheoryvaluationsmanifolds.2006}*{Proposition 3.1.2}.	Note that for any $x\in\R^n$, $D^n_x\mu:\calP(\R^n)\rightarrow\R$ defines a translation invariant valuation that is homogeneous of degree $n$ by \autoref{proposition:PropertiesLambdaDifferential}. \autoref{proposition:VolumeCharacterizationHomMeasurable} thus shows that $D^n_x\mu$ is a multiple of the Lebesgue measure $\vol_n$ on $\R^n$. Since this holds for any $x\in\R^n$, we obtain a function $c:\R^n\rightarrow\R$ such that 
		\begin{align*}
			\frac{d^n}{dt^n}\Big|_0\mu(tP+x)=D^n_x\mu(P)=c(x)\vol_n(P)\quad\text{for all}~P\in \calP(\R^n),~x\in\R^n.
		\end{align*}
		However, for a fixed polytope $P\in\calP(\R ^n)$, the left hand side is a smooth function in $x\in\R^n$ due to \autoref{theorem:DifferentiabilityAffineSmooth}. Thus $c$ is a smooth function. If we let $\tilde{\mu}$ denote the valuation given by the smooth measure on $\R^n$ with density $\frac{c}{n!}$, then it is easy to see that $\tilde{\mu}\in W_n(\R ^n)$ with
		\begin{align*}
			D^n_x\tilde{\mu}=c(x)\vol_n.
		\end{align*}
		In particular, $\mu-\tilde{\mu}\in W_{n+1}(\R^n)$. Since this space is trivial by \autoref{proposition:FiltrationTerminates}, $\mu=\tilde{\mu}$, so $\mu$ is a smooth measure.
	\end{proof}
	
	The next result may be seen as a refined version of \autoref{proposition:FiltrationTerminates}. It provides some uniform decay estimates for certain compact families of simplices.
	\begin{lemma}\label{lemma:EstimateSimplexVanishing}
		Let $\mu\in W_k(\R^n)$ and assume that $D^k_0\mu=0$. Fix a simplex $\Delta_0$ with vertex at the origin and a compact subset $A\subset \GL(n,\R)$. Consider the set $S(\Delta_0,A):=\{\Delta:\Delta=tg\Delta_0, t\in (0,1), g\in A\}$.  Then for every $\epsilon>0$ there exists $\delta>0$ such that for every simplex $\Delta\in S(\Delta_0,A)$ with $\diam\Delta<\delta$,
		\begin{align*}
			|\mu(\Delta)|\le C\epsilon \diam(\Delta)^k.
		\end{align*}
		for a constant $C>0$ depending on $\Delta_0$ and $A$ only.
	\end{lemma}
	\begin{proof}
		Consider the function $f:[0,1)\times \GL(n,\R)\rightarrow \R$ given by
		\begin{align*}
			f(t,g)=\mu(tg\Delta_0).
		\end{align*}
		Then $f$ is smooth by \autoref{theorem:DifferentiabilityAffineSmooth} and satisfies 
		\begin{align*}
			\frac{d^i}{dt^i}\Big|_0f(t,g)=0
		\end{align*}
		for all $0\le i\le k$ by assumption. Taylor's formula with remainder implies
		\begin{align*}
			f(t,g)=\int_0^t \frac{(t-s)^{k-1}}{(k-1)!}\partial_s^kf(s,g)ds=t^k\int_0^1 \frac{(1-s)^{k-1}}{(k-1)!}\partial_s^kf(ts,g)ds.
		\end{align*}
		Since the function $[0,1]\times A\rightarrow \R$, $(t,g)\mapsto \partial_s^kf(s,g)$ is continuous and $A$ is compact, ir is uniformly continuous. Moreover, it vanishes in $s=0$ for all $g\in A$. Thus, for every $\epsilon>0$, we find $\delta>0$ such that
		\begin{align*}
			|\partial_s^kf(s,g)|<\epsilon\quad\text{for all}~(s,g)\in[0,\delta]\times A.
		\end{align*}
		In particular, for $(t,g)\in [0,\delta]\times A$, we have
		\begin{align*}
			|f(t,g)|\le t^k\int_0^1 \frac{(1-s)^{k-1}}{(k-1)!}|\partial_s^kf(ts,g)|ds\le t^k\epsilon \int_0^1\frac{(1-s)^{k-1}}{(k-1)!}ds.
		\end{align*}
		Thus, for $t\in [0,\delta]$ and $g\in A$, we have
		\begin{align*}
			|\mu(tg\Delta_0)|\le \epsilon t^k.
		\end{align*}
		Since $A$ is compact, there exists a constant $C>1$ such that 
		\begin{align*}
			\frac{1}{C}\diam (\Delta_0)\le \diam (g\Delta_0)\quad\text{for all}~g\in A,
		\end{align*} 
		In particular, if $\Delta\in S(\Delta_0,A)$ satisfies $\diam (\Delta)<\frac{1}{C}\diam (\Delta_0)\delta$,  and $t\in[0,1]$, $g\in A$ are chosen such that $\Delta=tg\Delta_0$, then
		\begin{align*}
			t=\frac{\diam(\Delta)}{\diam(g\Delta_0)}< \frac{\frac{1}{C}\diam (\Delta_0)\delta}{\frac{1}{C}\diam (\Delta_0)}=\delta.
		\end{align*}
		In this case, we therefore obtain the inequality
		\begin{align*}
			|\mu(\Delta)|=|\mu(tg\Delta_0)|\le \epsilon t^k=\epsilon \frac{\diam(\Delta)^k}{\diam(g\Delta_0)^k}\le \epsilon \frac{C^k\diam(\Delta)^k}{\diam(\Delta_0)^k}.
		\end{align*}
		The claim follows by replacing $\delta$ by $\frac{1}{C}\diam (\Delta_0)\delta$.
	\end{proof}

	\begin{corollary}\label{corollary:ReformulationEstimateSimplex}
		Let $\mu\in W_k(\R^n)$ with $D^k_0\mu=0$, and assume that $S\subset \calP(\R^n)$ is a compact family of full dimensional simplices with vertex at the origin. Then the following holds: There exists a constant $C>0$ depending on $S$ only such that for every $\epsilon>0$, there exists $\delta>0$ such that for every simplex $\Delta\in \calP(\R^n)$ with $t\Delta\in S$ for some $t>0$ and $\diam (\Delta)<\delta$, the following inequality holds:
		\begin{align*}
			|\mu(\Delta)|\le C\epsilon \diam(\Delta)^k.
		\end{align*}
	\end{corollary}
	\begin{proof}
		Since $S$ is compact and every simplex is full dimensional, $\tilde{S}:=\{\vol(\Delta)^{-\frac{1}{n}}\Delta:\Delta\in S\}$ is well defined and compact as well, and every element has a vertex at the origin. If we fix an element $\Delta_0\in \tilde{S}$, then every element in $\tilde{S}$ is given by $g\Delta_0$ for some $g\in \SL(n,\R)$. The set $A:=\{g\in \SL(n,\R):g\Delta_0\in \tilde{S} \}$ is compact since $g$ is determined by the images of the vertices different from $0$ up to a permutation, so the map $A\rightarrow \tilde{S}$, $g\mapsto g\Delta_0$ defines a continuous $n!$-fold cover. Thus the claim follows by applying \autoref{lemma:EstimateSimplexVanishing} to the compact set $A\subset\SL(n,\R)$ and the simplex $\Delta_0$.
	\end{proof}

\section{Isometry invariant valuations on the sphere}
	\subsection{$\GL(n+1,\R)$-smooth valuations on the sphere}
		\label{section:glSmooth}
		Let $\CV(\S^n)$ denote the space of all continuous valuations $\mu:\calP(\S^n)\rightarrow\R$. Similar to the notion of affine smooth valuations on $\calP(\R ^n)$, let us call $\mu\in \CV(\S^n)$ a $\GL(n+1,\R)$-smooth valuation if the map
		\begin{align*}
			\GL(n+1,\R)&\rightarrow \R\\
			g&\mapsto \mu(g(P))
		\end{align*}
		is smooth for every $P\in\calP(\S^n)$, where $g(P)$ denotes the action of $\GL(n+1,\R)$ on $\calP(\S^n)$ from \autoref{section:sphericalPolytopes}. Note that this is again a pointwise property. Let us denote the subspace of $\GL(n+1,\R)$-smooth valuations by $\CV(\S^n)^{sm}$.
		
		\begin{remark}
			It is not difficult to check that the spherical intrinsic volumes are $\GL(n+1,\R)$-smooth (since they are smooth valuations in the sense of Alesker, compare \cite{AleskerTheoryvaluationsmanifolds.2006}).
		\end{remark}
		
		We require the following approximation result.
		
		\begin{corollary}\label{corollary:DensitySmoothInvariant}
			Let $\CV(\S^n)^{\SO(n+1)}$ denote the subspace of $\SO(n+1)$-invariant valuations. Then every $\mu\in \CV(\S^n)^{\SO(n+1)}$ is the pointwise limit of a sequence in $\CV(\S^n)^{\SO(n+1)}\cap \CV(\S^n)^{sm}$. Moreover, if $\mu$ is simple, then the sequence can be chosen to consists of simple valuations.
		\end{corollary}
		\begin{proof}
			Choose a smooth approximation of the identity $\phi_j\in C^\infty_c(\GL(n+1,\R))$. Given $\mu\in \CV(\S^n)$, it is easy to check that
			\begin{align*}
				\mu_j(P):=\int_{\GL(n+1,\R)} \phi_j(g)\mu(g^{-1}(P))dg
			\end{align*}
			defines an element in $\CV(\S^n)^{sm}$ and that we have $\lim_{j\rightarrow\infty}\mu_j(P)=\mu(P)$ pointwise (in fact, locally uniformly, but we will not need this stronger property). If $\mu$ is in addition $\SO(n+1)$-invariant, then we may use the left invariance of the Haar measure on $\GL(n+1,\R)$ and define
			\begin{align*}
				\tilde{\mu}_j(P):=&\int_{\SO(n+1)}\mu_j(h^{-1}P)dh\\
				=&\int_{\GL(n+1,\R)} \left[\int_{\SO(n+1)}\phi_j(h^{-1}g)dh\right]\mu(g^{-1}(P))dg\\
				=&\int_{\GL(n+1,\R)} \tilde{\phi}_j(g)\mu(g^{-1}(P))dg
			\end{align*}
			for a function $\tilde{\phi}_j\in C^\infty_c(\GL(n+1,\R))$. As in the case of $\mu_j$, one easily checks that $\tilde{\mu}_j\in \CV(\S^n)^{sm}$. Moreover, $\tilde{\mu}_j$ is obviously $\SO(n+1)$-invariant. Since $\mu$ is $\SO(n+1))$-invariant, one easily checks that the sequence $(\tilde{\mu}_j(P))_j$ converges to $\mu(P)$ for every $P\in\calP(\S^n)$. This shows the first claim. Note that the sequence constructed this way consists of simple valuations if $\mu$ is simple, since  in this case the integrands vanish identically for lower dimensional polytopes.
		\end{proof}

	\subsection{Relation to affine smooth valuations}

		For $u\in \S^n$, consider the gnomonic projection $G_{u}:\S^n_+(u)\rightarrow u^\perp$. Given $\mu\in \CV(\S^n)$, we define $G_{u*}\mu: \calP(u^\perp)\rightarrow\R$ by 
		\begin{align*}
			[G_{u*}\mu](P)=\mu(G_{u}^{-1}(P))\quad\text{for}~P\in\calP(u^\perp).
		\end{align*}
		Note that this is well-defined due to \autoref{lemma:gnomonicHomeomorphismPolytopes}.
		\begin{lemma}
			Let $u\in\S^n$. If $\mu\in\CV(\S^n)$, then $G_{u*}\mu\in\CV(\R^n)$. If $\mu$ is in addition $\GL(n+1,\R)$-smooth, then $G_{u*}\mu$ is affine smooth.
		\end{lemma}
		\begin{proof}
			Note that $G_{u*}\mu$ is clearly a valuation, so the first claim follows from \autoref{lemma:gnomonicHomeomorphismPolytopes}. The second claim follows directly from \autoref{lemma:EquivarianceAction}.
		\end{proof}
		
		For $x\in \R^n$, \autoref{lemma:EquivarianceAction} implies for $P\in\calP(u^\perp)$,
		\begin{align}
			\label{eq:relationEquivariance}
			[G_{u*}\mu](tP+x)=\mu(g^{u,x}_t(G_u^{-1}(P))).
		\end{align}
		The next lemma is the key result enabling the proofs in the following sections. It relies on the uniform differentiability on compact families of simplices from \autoref{corollary:ReformulationEstimateSimplex} and the convergence of the rescaled simplices from \autoref{lemma:LimitCoordinateChangeSimplex}.
		
		\begin{lemma}\label{lemma:KeyLemmaSimpleInvariant}
			Let $\mu\in \CV(\S^n)^{sm}$ be $\SO(n+1)$-invariant and simple and assume that $G_{u*}\mu\in W_k(u^\perp)$ for all $u\in \S^n$. If $D^k_0(G_{u_0*}\mu)=0$ for some $u_0\in \S^n$, then $G_{u*}\mu\in W_{k+1}(u^\perp)$ for every $u\in\S^n$.
		\end{lemma}
		\begin{proof}
			Due to Eq.~\eqref{eq:relationEquivariance}, we need to show that for every $u\in \S^n$, $x\in u^\perp$,
			\begin{align*}
				\frac{d^k}{dt^k}\Big|_0[G_{u*}\mu](tP+x)=	\frac{d^k}{dt^k}\Big|_0\mu(g^{u,x}_t(G_u^{-1}(P)))=0
			\end{align*}
			for all $P\in\calP(u^\perp)$ . First note that due to \autoref{corollary:ScalingMapsTransition}, we have for $V\in\SO(n+1)$,
			\begin{align*}
				D^k_x(G_{u*}\mu)[P]=&\frac{d^k}{dt^k}\Big|_0\mu(g^{u,x}_t(G_u^{-1}(P)))=\frac{d^k}{dt^k}\Big|_0\mu(V^{-1}Vg^{u,x}_tV^{-1}V(G_u^{-1}(P)))\\
				=&\frac{d^k}{dt^k}\Big|_0\mu(g^{Vu,Vx}_tV(G_u^{-1}(P))),
			\end{align*}
			as $\mu$ is $\SO(n+1)$-invariant. If we choose $V\in\SO(n+1)$ with $Vu=u_0$ and consider $x=0$, this implies
			\begin{align}
				\label{eq:DifferentialVanishInZeroEverywhere}
				\begin{split}
					D^k_0(G_{u*}\mu)[P]=&\frac{d^k}{dt^k}\Big|_0\mu(g^{Vu,0}_tG_{u_0}^{-1}G_{u_0}V(G_u^{-1}(P)))\\
					=&D^k_0(G_{u_0*}\mu)[G_{u_0}V(G_u^{-1}(P))]=0.
				\end{split}
			\end{align}
			Now let $\Delta\in \calP(\S^n_+(u_0))$ be a simplex with vertex $u_0\in \Delta$. For $v\in \calP(\S^n_+(u_0))$, the simplices in $v^\perp$ defined for $t>0$ small enough by
			\begin{align*}
				\Delta_t:=\frac{1}{t}\left[G_{v} (g_t^{u_0,G_{u_0}(v)}(\Delta_0))\right]
			\end{align*}
			converge to a full dimensional simplex $\Delta\in v^\perp$ with vertex $0\in \Delta$ by \autoref{lemma:LimitCoordinateChangeSimplex}. In particular, we may fix $t_0>0$ such that
			\begin{align*}
				S=\{\Delta_t:t\in(0,t_0]\}\cup \{\Delta\}
			\end{align*}
			is a compact family of full dimensional simplices with vertex in $0$. Since Eq.~\eqref{eq:DifferentialVanishInZeroEverywhere} holds for $u=v$,  \autoref{corollary:ReformulationEstimateSimplex} shows that for every $\epsilon>0$ there is $\delta>0$ such that
			\begin{align*}
				|G_{v*}\mu (t\Delta_t)|\le C\epsilon \diam\left(\Delta_t\right)^k t^k
			\end{align*}
			for all $t$ such that $\diam\left(t\Delta_t\right)<\delta$, where the constant $C$ is independent of $\epsilon$ and $\delta$. Since the diameter of $\Delta_t$ converges to $\diam(\Delta)$, this holds for all $t\in (0,\delta_0)$ for some $\delta_0>0$. Moreover, the diameters are bounded, so we find a constant $\tilde{C}>0$ independent of $\epsilon$ and $\delta_0$ such that
			\begin{align}
				\label{eq:EstimateSmallSimplex}
				|G_{v*}\mu (t\Delta_t)|\le \tilde{C}\epsilon t^k\quad \text{for}~t\in (0,\delta_0).
			\end{align}
			By assumption, $G_{u_0*}\mu\in W_k(u^\perp)$. In particular, we have
			\begin{align*}
				D^k_{G_{u_0}(v)}(G_{u_0*}\mu)[ G_{u_0}(\Delta_0)]=&k!\lim\limits_{t\rightarrow0}\frac{(G_{u_0*}\mu)[ tG_{u_0}(\Delta_0)+G_{u_0}(v)]}{t^k}\\
				=&k!\lim\limits_{t\rightarrow0}\frac{\mu(g_t^{u_0,G_{u_0}(v)}(\Delta_0))}{t^k}\\
				=&k!\lim\limits_{t\rightarrow0}\frac{(G_{v*}\mu)(G_v(g_t^{u_0,G_{u_0}(v)}(\Delta_0)))}{t^k}\\
				=&k!\lim\limits_{t\rightarrow0}\frac{[G_{v*}\mu](t\Delta_t)}{t^k}.
			\end{align*}
			From Eq.~\eqref{eq:EstimateSmallSimplex} we thus obtain
			\begin{align*}
				|	D^k_{G_{u_0}(v)}(G_{u_0*}\mu)[ G_{u_0}(\Delta_0)]|\le k!\tilde{C}\epsilon,
			\end{align*}
			where the constant does not depend on $\epsilon>0$. Since this holds for all $\epsilon>0$, this shows that $D^k_{G_{u_0}(v)}(G_{u_0*}\mu)$ vanishes in all full dimensional simplices of the form $G_{u_0}(\Delta_0)$ for any full dimensional simplex $\Delta_0$ in $\S^n_+(u_0)$ with vertex $u_0\in \Delta_0$. However, every full dimensional simplex in $u_0^\perp$ is a translate of a simplex of this form, so since $D^k_{G_{u_0}(v)}(G_{u_0*}\mu)$ is translation invariant by \autoref{proposition:PropertiesLambdaDifferential}, it vanishes in all full dimensional simplices. Since $\mu$ is a simple valuation, so is $D^k_{G_{u_0}(v)}(G_{u_0*}\mu)$, so by triangulating a given polytope, we see that  $D^k_{G_{u_0}(v)}(G_{u_0*}\mu)$ vanishes identically. Since this holds for all $v\in \S^n_+(u_0)$ and the gnomonic projection $G_{u_0}$ is bijective, this shows that $G_{u_0*}\mu\in W_{k+1}(u_0^\perp)$, compare \autoref{remark:VanishingDerivativeFiltration}. As $\mu$ is $\SO(n+1)$-invariant, we obtain $G_{u*}\mu\in W_{k+1}(u^\perp)$ for any $u\in \S^n$.
		\end{proof}
		\begin{remark}
			The previous results serves as a substitute for a proper chain rule: If we could interpret the map $P\mapsto\mu(g_t^{u_0,G_{u_0}(v)}(P))$ as a properly differentiable functional, we would expect that its differential should be given by both
			\begin{align*}
				D^k_{G_{u_0}(v)}[G_{u_0*}\mu](G_{u_0}(P))\quad \text{and}\quad D^k_{0}[G_{v*}\mu](F(P)),
			\end{align*}
			where the right hand side reflects the chain rule, i.e. $F$ is the derivative of $(t,y)\mapsto G_v(g_t^{u_0,G_{u_0}(v)}(y))$ with respect to $t$ in $t=0$, compare \autoref{lemma:LimitCoordinateChangeSimplex}, which would directly imply that the derivative vanishes. However, our differentiability property is too weak to allow for this interpretation, which is why we need to use more elementary tools to circumvent this problem. Informally, affine smooth valuations only allow for a restricted family of directional derivatives that are not compatible with the relevant changes of coordinates through gnomonic projections (which are fractional linear maps).
		\end{remark}

	\subsection{Characterization of $\SO(n+1)$-invariant valuations}
		In this section we combine the previous results to prove our main theorems. We begin with a version of \autoref{maintheorem:HadwigerSphereSimple} for $\GL(n+1,\R)$-smooth valuations.
		
		\begin{theorem}\label{theorem:ClassificationSimpleSmooth}
			Let $\mu\in \CV(\S^n)^{sm}$ be $\SO(n+1)$ invariant and simple. Then $\mu$ is a multiple of the spherical Lebesgue measure.
		\end{theorem}
		\begin{proof}
			We will show by induction on $1\le k\le n$ that $G_{u*}\mu\in W_k(u^\perp)$ for all $u\in \S^n$. First, note that the claim holds for $k=1$ since $\mu$ is simple and thus vanishes on singletons. Thus assume that $k\le n-1$ and that $G_{u*}\mu\in W_k(u^\perp)$. We fix $u_0\in \S^n$. Then 
			\begin{align*}
				D^k_0(G_{u_0*}\mu):\calP(u_0^\perp)\rightarrow \R
			\end{align*}
			is translation invariant, measurable and $k$-homogeneous by \autoref{proposition:PropertiesLambdaDifferential}. By construction, it is also simple, and \autoref{lemma:EquivarianceAction} implies that it is $\SO(u^\perp)$-invariant (since we are considering the map at the origin in $u^\perp$). Since $k\le n-1$, \autoref{proposition:simpleValuationsVanish} shows that it vanishes identically. Thus \autoref{lemma:KeyLemmaSimpleInvariant} implies that $G_{u*}\mu\in W_{k+1}(u^\perp)$ for all $u\in\S^n$. This complete the induction.\\			
			We therefore obtain that $G_{u*}\mu\in W_n(u^\perp)$ for every $u\in \S^n$. By \autoref{proposition:classificationWnDensity}, $G_{u*}\mu$ is thus given by a smooth measure on $u^\perp$. We may pull this measure back to $\S^n_+(u)$ to see that the restriction of $\mu$ to $\calP(\S^n_+(u))$ is given by a smooth measure for all $u\in \S^n$. Thus these measures have to coincide on $\S^n_+(u)\cap \S^n_+(v)$ for all $u,v\in \S^n$, and we may glue them to a smooth measure on $\S^n$. Since $\mu$ is $\SO(n+1)$-invariant, so is this measure, so it is a multiple of the spherical Lebesgue measure. 
		\end{proof}
		
		\begin{proof}[Proof of \autoref{maintheorem:HadwigerSphereSimple}]
			If $\mu\in\CV(\S^n)^{\SO(n+1)}$ is simple, then we may choose a sequence $(\mu_j)_j$ of $\GL(n+1,\R)$-smooth valuations in $\CV(\S^n)^{\SO(n+1)}$ converging pointwise to $\mu$ by \autoref{corollary:DensitySmoothInvariant}. Moreover, we may assume that these valuations are simple. Thus \autoref{theorem:ClassificationSimpleSmooth} implies that $\mu_j=c_j V_n$ is a multiple of the spherical Lebesgue measure. Take $P_0\in\calP(\S^n)$ with $V_n(P_0)\ne 0$. Then 
			\begin{align*}
				\lim_{j\rightarrow\infty}c_j=\lim_{j\rightarrow\infty}\frac{\mu_j(P_0)}{V_n(P_0)}= \frac{\mu(P_0)}{V_n(P_0)}
			\end{align*}
			since the sequence $(\mu_j(P_0))_j$ converges to $\mu(P_0)$. In particular, for every $P\in\calP(\S^n)$,
			\begin{align*}
				\mu(P)=\lim\limits_{j\rightarrow\infty}\mu_j(P)=\lim\limits_{j\rightarrow\infty}c_j V_n(P)=\frac{\mu(P_0)}{V_n(P_0)}V_n(P).
			\end{align*}
			Thus $\mu$ is a multiple of the spherical Lebesgue measure.
		\end{proof}
		\autoref{maintheorem:HadwigerSphere} follows from \autoref{maintheorem:HadwigerSphereSimple} with a standard argument, which we include for completeness.
		\begin{proof}[Proof of \autoref{maintheorem:HadwigerSphere}]
			We use induction on the dimension $n$. For $n=1$ and $\mu\in \CV(\S^1)^{\SO(2)}$, the valuation $\tilde{\mu}=\mu-\mu(\{u_0\})V_0$ belongs to $\CV(\S^1)^{\SO(2)}$ and vanishes on singletons. Thus it is simple, and therefore  $\tilde{\mu}=c V_1$ is a multiple of the spherical Lebesgue measure by \autoref{maintheorem:HadwigerSphereSimple}, so $\mu=\mu(\{u_0\})V_0+cV_1$ is a linear combination of the spherical intrinsic volumes.\\
			
			Assume that the claim holds for all dimensions up to $n-1$. If $\mu\in \CV(\S^n)^{\SO(n+1)}$, consider the great sphere $\S(e_{n+1}^\perp)=\S^n\cap e_{n+1}^\perp$. Then the restriction of $\mu$ to $\calP(\S(e_{n+1}^\perp))$ defines a continuous and $\SO(e_{n+1}^\perp)$-invariant valuation. In particular, the induction hypothesis implies that there exist $c_0,\dots,c_{n-1}\in\R$ such that the restriction of $\mu$ to $\S(e_{n+1}^\perp)$ is given by $\sum_{j=0}^{n-1}c_j V_j|_{\S(e_{n+1}^\perp)}$. Note that this implies that 
			\begin{align*}
				\tilde{\mu}:=\mu-\sum_{j=0}^{n-1}c_j V_j
			\end{align*}
			vanishes on $\calP(\S(e_{n+1}^\perp))$, so since $\mu$ is $\SO(n+1)$ invariant, $\tilde{\mu}$ is a simple valuation. Thus $\tilde{\mu}$ satisfies the assumptions of \autoref{maintheorem:HadwigerSphereSimple}, so $\tilde{\mu}=c_nV_n$ is a multiple of the spherical Lebesgue measure, and we obtain
			\begin{align*}
				\mu=\tilde{\mu}+\sum_{j=0}^{n-1}c_j V_j=\sum_{j=0}^nc_jV_j.
			\end{align*}
			This completes the proof.
		\end{proof}

\bibliographystyle{plain}
\bibliography{../../library/library.bib}

\Addresses
	
\end{document}